\documentclass[reqno]{amsart}

\usepackage{amssymb}
\usepackage{enumitem}
\usepackage{mathdots}
\usepackage[initials, lite, msc-links, nobysame]{amsrefs}

\newcommand{\cnt}{\mathcal{Z}}
\newcommand{\grp}{\mathcal{P}}
\newcommand{\inv}{^{-1}}
\newcommand{\iso}{\cong}
\newcommand{\K}{\Bbbk}
\newcommand{\M}{\mathcal{M}}
\newcommand{\NN}{\mathbb{N}}
\newcommand{\niso}{\ncong}
\newcommand{\PP}{\mathbb{P}}
\newcommand{\sfsim}{\sim_\mathrm{sf}}
\newcommand{\stf}{\mathrm{sf}}

\DeclareMathOperator\Aut{Aut}
\DeclareMathOperator\GL{GL}
\DeclareMathOperator\gk{gk}
\DeclareMathOperator\gld{gld}
\DeclareMathOperator\gr{gr}
\DeclareMathOperator\id{id}
\DeclareMathOperator\Stab{Stab}
\DeclareMathOperator\SL{SL}
\DeclareMathOperator\Sp{Sp}

\newcommand{\xx}{R_{x^2}}
\newcommand{\yx}{R_{yx}}
\newcommand{\qp}{\mathcal{O}_q(\K^2)}
\newcommand{\qpp}{\mathcal{O}_p(\K^2)}
\newcommand{\fwa}{A_1(\K)}
\newcommand{\wa}{A_1^q(\K)}
\newcommand{\wap}{A_1^p(\K)}
\newcommand{\jp}{\mathcal{J}}
\newcommand{\sjp}{\mathcal{J}_1}
\newcommand{\os}{\mathcal{S}}
\newcommand{\env}{\mathfrak{U}}
\newcommand{\envv}{\mathfrak{V}}
\newcommand{\kx}{\K[x]}
\newcommand{\sxx}{R_{x^2-1}}

\newcommand{\hxx}{H(\xx)}
\newcommand{\hyx}{H(\yx)}
\newcommand{\hqp}{H(\qp)}
\newcommand{\hqpp}{H(\qpp)}
\newcommand{\hwa}{H(\wa)}
\newcommand{\hwap}{H(\wap)}
\newcommand{\hjp}{H(\jp)}
\newcommand{\hsjp}{H(\sjp)}
\newcommand{\hos}{H(\os)}
\newcommand{\henv}{H(\env)}
\newcommand{\henvv}{H(\envv)}
\newcommand{\hkx}{H(\kx)}
\newcommand{\hsxx}{H(\sxx)}

\newcommand{\bp}{\mathbf{p}}
\newcommand{\bq}{\mathbf{q}}
\newcommand{\qnp}{\mathcal{O}_{\bp}(\K^n)}
\newcommand{\qnq}{\mathcal{O}_{\bq}(\K^n)}

\newtheorem{thm}{Theorem}[section]
\newtheorem{cor}[thm]{Corollary}
\newtheorem{lem}[thm]{Lemma}
\newtheorem{prop}[thm]{Proposition}

\newtheorem{defn}[thm]{Definition}

\numberwithin{equation}{section}

\title{Two-generated algebras and standard-form congruence}
\author{Jason Gaddis}
\address{Wake Forest University, Department of Mathematics, P. O. Box 7388, Winston-Salem, NC 27109} 
\email{jdgaddis@gmail.com}
\thanks{}

\date{}

\begin{document}

\begin{abstract}
Matrix congruence can be used to mimic linear maps 
between homogeneous quadratic polynomials in $n$ variables. 
We introduce a generalization, 
called standard-form congruence, 
which mimics affine maps between non-homogeneous quadratic polynomials. 
Canonical forms under standard-form congruence for three-by-three matrices are derived. 
This is then used to give a classification of algebras defined by 
two generators and one degree two relation. 
We also apply standard-form congruence to classify homogenizations of these algebras.

\end{abstract}

\maketitle

\section{Introduction}

Let $\K$ be an algebraically closed field of characteristic zero. 
All algebras are $\K$-algebras and all isomorphisms are as $\K$-algebras.
We denote by $\M_n(\K)$ the ring of $n \times n$ matrices over $\K$.
We denote the center of an algebra $A$ by $\cnt(A)$.

Our interest is in algebras $A$ defined as a factor of the free algebra 
on two degree one generators by a single degree two relation, i.e., 
\begin{align}\label{form}
  A = \K\langle x,y \mid f \rangle, \deg(f)=2.
\end{align}
In case $f$ is homogeneous, the classification of such algebras is well-known 
(see, e.g., \cite{smith}). 
The polynomial $f$ can be represented by a $2 \times 2$ matrix and matrix congruence 
corresponds to linear isomorphisms between homogeneous algebras. 
Hence, canonical forms for matrices in $\M_2(\K)$ give a maximal list of algebras to consider. 
One must verify that there are no non-linear isomorphisms between the remaining algebras. 
This can be accomplished by considering ring-theoretic properties, 
resulting in four types of algebras: 
the quantum planes $\qp$, the Jordan plane $\jp$, $\yx$, and $\xx$.

We give a method for extending this idea to algebras in which $f$ is not necessarily homogeneous. 
In Section \ref{stfc}, we develop a modified version of matrix congruence called \textit{standard-form congruence}. Canonical forms in $\M_3(\K)$ under standard-form congruence are determined in Section \ref{canforms}. 
These forms are in near 1-1 correspondence with isomorphism classes of algebras of the form \eqref{form}. 
This leads to the following theorem.

\begin{thm}
\label{classification}
Suppose $A \iso \K\langle x,y \mid f \rangle$ 
where $f$ is a polynomial of degree two. 
Then $A$ is isomorphic to one of the following algebras:
\begin{align*}
  &\qp, f=xy-qyx ~ (q \in \K^\times), 
  	& ~ & \wa, f=xy-qyx-1 ~ (q \in \K^\times), \\
  &\jp,  f=yx-xy+y^2, 
  	& ~ & \sjp, f=yx-xy+y^2+1, \\
  &\env, f=yx-xy+y, 
  	& ~ & \kx, f= x^2+y, \\
  &\xx, f=x^2, 
  	& ~ & \sxx, f=x^2-1,\\
  &\yx, f=yx, 
  	& ~ & \os, f=yx-1.
\end{align*}
Furthermore, the above algebras are pairwise non-isomorphic, except \[\qp \iso \mathcal{O}_{q\inv}(\K^2) \text{ and } \wa \iso A_1^{q\inv}(\K).\]\end{thm}

Many of these algebras are well-known. 
The algebras $\wa$ are the \textit{quantum Weyl algebras},
$\env$ the enveloping algebra of the non-abelian two-dimensional solvable Lie algebra, 
and $\sjp$ the \textit{deformed Jordan plane}.
This list slightly contradicts that given in \cite{smith} 
since $\os$ and $\sjp$ both have Gelfand-Kirillov (GK) dimension two.
We prove this theorem in Section \ref{dist}.

We define one additional algebra,
 \[ \envv = \K\langle x,y \mid yx-xy+y^2+x\rangle.\]
This algebra is not included in Theorem \ref{classification} 
because it is isomorphic to $\env$ (Proposition \ref{envv}).

As a second application of standard-form congruence,
we consider a related class of algebras,
\begin{align}
\label{hform} 	&H = \K\langle x,y,z \mid xz-zx, yz-zy, f \rangle, \\
\notag  		&f \in \K\langle x,y,z \rangle, \;\; \deg(f)=2,\;\; f \notin \K[z], \;\; f \text{ homogeneous}.
\end{align}
Algebras of form \eqref{hform} may be 
regarded as \textit{homogenizations} 
of those of form \eqref{form}. 
In Section \ref{homog}, 
we prove (Theorem \ref{hclass}) that canonical forms under sf-congruence are in 1-1 correspondence 
with isomorphism classes of algebras of the form \eqref{hform}.

This result may be framed in terms of \textit{(Artin-Schelter) regular algebras}.
We refer the reader to \cite{NVZ} for undefined terms.
If $H$ is a global dimension three regular algebra,
then $H$ is associated to a point scheme $C \subset \PP^2$ 
and an automorphism $\sigma$ of $C$.
Suppose  $C$ contains a line fixed by $\sigma$ (Type $S_1'$ regular algebras).
By \cite{NVZ}, Proposition 1.2, 
$H$ may be \textit{twisted} so that it is isomorphic to a 
$\K$-algebra on generators $x$, $y$, $z$ with 
defining relations
	\[ xz = zx, \;\;\; yz = zy, \;\;\; h=0,\]
where $h$ is one of the following polynomials
\begin{itemize}
	\item[(I)] $yx-xy + y^2 + z(\alpha x + \beta y + \gamma z)$,
	\item[(II)] $xy-qyx+z(\alpha x + \beta y + \gamma z)$,
	$q \in \K^\times$,
\end{itemize}
for some $\alpha,\beta,\gamma \in \K$.
Hence, Theorem \ref{hclass} provides a refinement on
this classification.

\section{Congruence}
\label{cpmg}

Let $f = a x^2 + b xy + c yx + d y^2$, $a,b,c,d \in \K$. 
By a slight abuse of notation,
  \[f = 
  	\begin{pmatrix}x & y\end{pmatrix} 
  	\begin{pmatrix}a & b \\ c & d\end{pmatrix}
   	\begin{pmatrix}x \\ y\end{pmatrix}.\] 
Hence, we can represent any homogeneous quadratic polynomial by an element of $\M_2(\K)$. 
If $A$ is of the form \eqref{form}, 
then $f$ is called a \textit{defining polynomial} 
for $A$ and the matrix corresponding to $f$ is called a \textit{defining matrix} for $A$. 
The map $\phi$ given by 
$x \mapsto p_{11}x + p_{12}y$ and 
$y \mapsto p_{21}x + p_{22}y$, $p_{ij}\in \K$, 
with $p_{11}p_{22}-p_{12}p_{21} \neq 0$ corresponds to a 
linear isomorphism between the algebras with defining polynomials $f$ and $\phi(f)$. 

Similarly, $M,M' \in \M_n(\K)$ are said to be \textit{congruent} 
and we write $M \sim M'$ if there exists $P \in \GL_n(\K)$ such that $P^T M P = M'$. 
Matrix congruence is an equivalence relation on the set $\M_n(\K)$. 
A \textit{canonical form} under congruence is a distinguished representative from an equivalence class.

When two defining matrices are congruent there is an linear map between the polynomials that they determine. 
In turn, the algebras with these defining polynomials are isomorphic. 
On the other hand, if there is a linear map between two defining polynomials, 
then the corresponding algebras are isomorphic. 
However, two such algebras can still be isomorphic even if there is no linear map between the defining polynomials. 
Thus, canonical forms for congruent matrices give us a maximal list of algebras to consider and we are then left to determine whether there are any other isomorphisms.

The Horn-Sergeichuk forms depend on three block-types which we henceforth refer to as \textit{HS-blocks}\index{HS-blocks},
\begin{align*}
  J_n(\lambda) &= \begin{pmatrix}
    \lambda   & 1       &       & 0 \\
          & \lambda  & \ddots   & \\
         &       & \ddots   &  1\\
      0  &      &      &  \lambda
  \end{pmatrix}, J_1(\lambda) = \begin{pmatrix}\lambda\end{pmatrix}, \\
  \Gamma_n &= \begin{pmatrix}
  0   &     &     &     &       & (-1)^{n+1} \\
    &     &     &     &  \iddots  &  (-1)^n \\
    &    &    &  -1  &  \iddots  & \\
    &    &  1  &  1   &      & \\
    & -1   &  -1  &    &      &  \\
  1   & 1    &      &     &       & 0
  \end{pmatrix}, \Gamma_1 = \begin{pmatrix}1\end{pmatrix}, \\
  H_{2n}(\mu) &= \begin{pmatrix}0 & I_n \\ J_n(\mu) & 0\end{pmatrix}, H_{2}(\mu) = \begin{pmatrix}0 & 1 \\ \mu & 0\end{pmatrix}.
\end{align*}

\begin{thm}[Horn, Sergeichuk \cite{hornserg}]
\label{hsthm}
Each square complex matrix is congruent to a direct sum, 
uniquely determined up to permutation of summands, 
of canonical matrices of the three types 
$J_n(0)$, $\Gamma_n$, and $H_{2n}(\mu)$, $\mu \neq 0,(-1)^{n+1}$. 
Moreover, $H_{2n}(\mu)$ is determined up to replacement of $\mu$ by $\mu\inv$.
\end{thm}

As a consequence of the previous theorem,
there are four HS-block types in dimension two: 
$\Gamma_1 \oplus J_1(0)$, $J_2(0)$, $\Gamma_2$, and $H_2(\mu)$.
We choose to use $J_2(0)^T$ in place of $J_2(0)$ and let
$q=\mu$ in $H_2(\mu)$.
These matrices are given explicitly as 
\begin{align}
\label{22forms}
	\begin{pmatrix}1 & 0 \\ 0 & 0\end{pmatrix}, 
	\begin{pmatrix}0 & 0 \\ 1 & 0\end{pmatrix}, 
	\begin{pmatrix}0 & -1 \\ 1 & 1\end{pmatrix}, 
  	\begin{pmatrix}0 & -1 \\ q & 0\end{pmatrix}, q \in \K^\times.
\end{align}
We denote these matrices by 
$M_{x^2}$, $M_{yx}$, $M_{\jp}$, and $M_q$, respectively.
It follows from Theorem \ref{hsthm}, 
or a quick computation, that $M_q \sim M_{q\inv}$.
Moreover, $M_p \sim M_q$ if and only if $p=q^{\pm 1}$ 
(see Corollary \ref{qp-cong}).

As it will be useful in the general (non-homogeneous) case
we compute the stabilizer groups for the matrices in \eqref{22forms}.
In general, these stabilizer groups correspond to some orthosymplectic group but, 
because some of the forms are degenerate, 
there are shifts in the dimension.
Any $M \in \M_n(\K)$ admits a unique decomposition $M=A+S$ where 
$A,S \in \M_n(K)$ with $S$ symmetric and $A$ antisymmetric.
Because congruence preserves symmetry (resp. antisymmetry),
then the stabilizer group of $M$ is equal to the intersection of
the stabilizer groups for $A$ and $S$.

\begin{prop}
\label{stabs}
Let $M$ be one of the matrices in \eqref{22forms}.
The group
	\[\Stab(M) = \{ P \in \GL_2(\K) \mid P^TMP=M \}\]
is described below.
\begin{align*}
\Stab\left( M_{x^2} \right) 
  	&= \left\lbrace \left. 
  	\begin{pmatrix}\pm 1 & 0 \\ r & s\end{pmatrix} 
  	\right| 	r,s \in \K^\times \right\rbrace, \\
\Stab\left( M_{yx} \right) 
	&= \left\lbrace \left.
		\begin{pmatrix}r & 0 \\ 0 & r\inv\end{pmatrix}
		\right| r \in \K^\times\right\rbrace, \\
\Stab\left( M_{\jp} \right) 
  	&= \left\lbrace \left. 
  		\pm \begin{pmatrix}1 & r \\ 0 & 1 \end{pmatrix}
  		\right| r \in \K^\times \right\rbrace, \\
\Stab\left( M_q \right) 
	&= \left\lbrace \left.
		\begin{pmatrix}r & 0 \\ 0 & r\inv\end{pmatrix}
		\right| r \in \K^\times \right\rbrace 
		~ (q \in \K^\times, q\neq \pm 1), \\
\Stab\left( M_{-1} \right) 
	&= \left\lbrace \left. 
		\begin{pmatrix}r & 0 \\ 0 & r\inv\end{pmatrix}, 
		\begin{pmatrix}0 & s \\ s\inv & 0\end{pmatrix} 
			\right| r,s \in \K^\times \right\rbrace, \\
\Stab\left( M_1 \right) &= \SL_2(\K).
\end{align*}
\end{prop}

\begin{proof}
Throughout, let $P \in \Stab(M)$ and write 
$P = \begin{pmatrix}a & b \\ c & d\end{pmatrix}$. 

The matrix $M_1$ corresponds to the standard basis non-degenerate dimension two alternating form. 
Thus, its stabilizer is $\Sp(2) \iso \SL_2(\K)$.

The matrix $M_{-1}$ corresponds to a non-standard basis 
non-degenerate dimension two symmetric form.
We have
\[ P^T M_{-1} P = \begin{pmatrix}2ac & ac+bd \\ ac+bd & 2bd\end{pmatrix}=\begin{pmatrix}0 & 1 \\ 1 & 0\end{pmatrix}.\] 
Hence, either $a=d=0$ or $b=c=0$ and the result follows.

In the case of $M_{x^2}$ we have
	\[ P^TM_{x^2}P = \begin{pmatrix}a^2 & ab \\ ab & b^2\end{pmatrix} = \begin{pmatrix}1 & 0 \\ 0 & 0\end{pmatrix}.\] 
Then $a=\pm 1$ and $b=0$.

We have
	\[ M_{\jp}=\begin{pmatrix}0 & 0 \\ 0 & 1\end{pmatrix} +\begin{pmatrix}0 & -1 \\ 1 & 0\end{pmatrix}.\] 
Thus, $\Stab(M_{\jp}) = \Stab(M_{x^2})^T \cap \Stab(M_1)$.

Finally, for $q \neq \pm 1$,
  \[ M_q = 
  	 \frac{q+1}{2}\begin{pmatrix}0 & -1 \\ 1 & 0\end{pmatrix}
  	+\frac{q-1}{2}\begin{pmatrix}0 & 1 \\ 1 & 0\end{pmatrix}.\]
Thus, $\Stab(M_q)=\Stab(M_1) \cap \Stab(M_{-1})$.

The case of $M_{yx}$ may be seen from the previous computation by letting $q=0$.
\end{proof}

\begin{cor}
\label{qp-cong}
Let $p,q \in \K^\times$. Then $M_p \sim M_q$ if and only if $p=q^{\pm 1}$.\end{cor}

\begin{proof}
Sufficiency is provided by Theorem \ref{hsthm}.
Suppose $M_p \sim M_q$ and choose $P \in \GL_2(\K)$ such that $M_p = P^TM_qP$. 
Write $P = \begin{pmatrix}a & b \\ c & d\end{pmatrix}$.
Then,
\begin{align}
\label{mqform}
	P^TM_qP = \begin{pmatrix}(q-1)ac & qbc-ad \\ qad-bc & (q-1)bd \end{pmatrix} = \begin{pmatrix}0 & -1 \\ q & 0\end{pmatrix}.
\end{align}
Observe that, if $q=1$, then $M_p=P^TM_1P=(ad-bc)M_1$,
and so $p=1$.
In general, we see by comparing $P^TM_qP$ to $M_p$ in \eqref{mqform}
that $ac=bd=0$.
Thus, either $b=c=0$ or $a=d=0$. 
In the first case, $M_p=P^TM_qP=(ad)M_q$, and so $p=q$. 
In the second case, $M_p=P^TM_qP = (-qbc)M_{q\inv}$, and so $p=q\inv$.
\end{proof}

\section{Standard Form Congruence}
\label{stfc}

In the non-homogeneous case, we write $f=ax^2+bxy+cyx+dy^2+\alpha x + \beta y + \gamma$, $a,b,c,d,\alpha,\beta,\gamma \in \K$. We can represent $f$ by a $3 \times 3$ matrix via the rule
  \[ f = \begin{pmatrix}x & y & 1\end{pmatrix}\begin{pmatrix}a & b & \alpha\\ c & d & \beta \\ 0 & 0 & \gamma\end{pmatrix}\begin{pmatrix}x \\ y \\ 1\end{pmatrix}.\]
We extend the terms \textit{defining polynomial} and \textit{defining matrix} as one would expect. However, our choice of defining matrix is not unique. One could define $f$ by
  \[ f = \begin{pmatrix}x & y & 1\end{pmatrix}\begin{pmatrix}a & b & 0\\ c & d & 0 \\ \alpha & \beta & \gamma\end{pmatrix}\begin{pmatrix}x \\ y \\ 1\end{pmatrix}.\]
Hence, it is necessary to fix a \textit{standard form} for the defining matrices of non-homogeneous polynomials. We restrict our attention to the following set,
  \[ G_3 = \left\lbrace \left. \begin{pmatrix}a_1 & a_2 & a_3 \\ b_1 & b_2 & b_3 \\ 0 & 0 & c \end{pmatrix} \right| \begin{pmatrix}a_1 & a_2 \\ b_1 & b_2\end{pmatrix} \neq 0\right\rbrace \subset \M_3(\K).\]
Every degree two polynomial has a unique corresponding matrix in $G_3$. Consider the matrix
  \[ M = \begin{pmatrix}m_{11} & m_{12} & m_{13} \\ m_{21} & m_{22} & m_{23} \\ m_{31} & m_{32} & m_{33}\end{pmatrix} \in \M_3(\K).\]
This corresponds to the polynomial
\begin{align*}
  f &= m_{11}x^2 + m_{12}xy + m_{13}x + m_{21}yx + m_{22} y^2 + m_{23} y + m_{31} x + m_{32} y + m_{33} \\
    &= m_{11}x^2 + m_{12}xy + m_{21}yx + m_{22}y^2 + (m_{13} + m_{31})x + (m_{23} + m_{32}) y + m_{33},
\end{align*}
which in turn corresponds to the matrix
  \[ \begin{pmatrix}m_{11} & m_{12} & m_{13} + m_{31} \\ m_{21} & m_{22} & m_{23} + m_{32} \\ 0 & 0 & m_{33}\end{pmatrix}.\]
Hence, we define a $\K$-linear map $\stf:\M_3(\K) \rightarrow G_3$ by
\begin{align*}
  \begin{pmatrix}m_{11} & m_{12} & m_{13} \\ m_{21} & m_{22} & m_{23} \\ m_{31} & m_{32} & m_{33}\end{pmatrix} &\mapsto
    \begin{pmatrix}m_{11} & m_{12} & m_{13} + m_{31} \\ m_{21} & m_{22} & m_{23} + m_{32} \\ 0 & 0 & m_{33}\end{pmatrix}.
\end{align*}
Let $p_{ij} \in \K$ and define a $\K$-linear map by
\begin{align}
\label{sf.lmap}
  \phi(x) = p_{11}x + p_{12}y + p_{13}, ~~\phi(y) = p_{21}x + p_{22}y + p_{23}, ~~\phi(1) = 1.
\end{align}
If $p_{11}p_{22}-p_{12}p_{21} \neq 0$, then $\phi$ defines an affine isomorphism between $\K\langle x,y \mid f \rangle$ and $\K\langle x,y \mid \phi(f) \rangle$. Thus, the matrices corresponding to affine isomorphisms of these algebras should be contained in the set
  \[ \grp_3 = \left\lbrace \begin{pmatrix}P_1 & P_2 \\ 0 & 1\end{pmatrix} \in \M_3(\K) \mid P_1 \in \GL_{2}(\K), P_2 \in \K^{2}\right\rbrace.\]

In general, we want a map that fixes the 
degree two part of a quadratic polynomial and adds the linear parts. 
We write $M \in \M_n(\K)$ in block form
\begin{align}\label{mform}
  M = \left\lbrace \begin{pmatrix}M_1 & M_2 \\ M_3^T & m\end{pmatrix} 
  \mid M_1 \in \M_{n-1}(\K), M_2,M_3 \in \K^{n-1},m \in \K \right\rbrace.
\end{align}
We call $M_1$ the \textit{homogeneous block} of $M$. 
Define the set
  \[ G_n = \left\lbrace \begin{pmatrix}M_1 & M_2 \\ 0 & m\end{pmatrix} 
  \in \M_n(\K) \mid 0 \neq M_1 \in \M_{n-1}(\K), M_2 \in \K^{n-1},m \in \K \right\rbrace.\]
Then define the map $\stf:\M_n \rightarrow G_n$ by
\begin{align}
\label{sfmap}
  \begin{pmatrix}M_1 & M_2 \\ M_3^T & m\end{pmatrix} 
  &\mapsto \begin{pmatrix}M_1 & M_2 + M_3 \\ 0 & m\end{pmatrix},
\end{align}
where the matrix is written according to \eqref{mform}. 
The matrices corresponding to affine isomorphisms of these algebras should be contained in the set
  \[ \grp_n = \left\lbrace \begin{pmatrix}P_1 & P_2 \\ 0 & 1\end{pmatrix} \in \M_n(\K) \mid P_1 \in \GL_{n-1}(\K), P_2 \in \K^{n-1}\right\rbrace.\]

\begin{prop}$\grp_n$ is a group.\end{prop}

\begin{proof}That $\grp_n$ contains the identity matrix is clear. Let $P,P' \in \grp_n$. Then
  \[ PP' = \begin{pmatrix}P_1 & P_2 \\ 0 & 1\end{pmatrix}\begin{pmatrix}P_1' & P_2' \\ 0 & 1\end{pmatrix} = \begin{pmatrix}P_1P_1' & P_1P_2'+P_2 \\ 0 & 1\end{pmatrix} \in \grp_n.\]
Since $P_1 \in \GL_{n-1}(\K)$, then we can set $P_1'=P_1\inv \in \GL_{n-1}(\K)$ and $P_2'=-P_1\inv P_2$. It is now clear from the above that $P'=P\inv$.\end{proof}  

Under ordinary matrix congruence, 
two matrices which are scalar multiples of each other are always congruent. 
However, if we restrict to $\grp_n$, that is no longer the case. 
Hence, in our modified definition of congruence, 
we set scalar multiple matrices to be congruent to each other.

\begin{defn}
\label{sfdef}
We say $M,N \in \M_n(\K)$ are \textbf{standard-form congruent} 
(sf-congruent) and write $M \sfsim N$ if there exist 
$P \in \grp_n$ and $\alpha \in \K^\times$ such that 
$\stf(M) = \alpha \cdot \stf(P^T N P)$.
\end{defn}

The next proposition shows that sf-congruence is a true extension of congruence.

\begin{prop}
\label{sub}
Let $M,N \in \M_n(\K)$ with homogeneous blocks $M_1,N_1$, respectively. 
If $M \sfsim N$, then $M_1 \sim N_1$.
\end{prop}

\begin{proof}
By hypothesis, $\stf(M) = \alpha \cdot \stf(P^T N P)$ for some $P \in \grp_n$, 
$\alpha \in \K^\times$. 
Then
\begin{align*}
    \begin{pmatrix}M_1 & M_2 \\ 0 & m\end{pmatrix} 
      &= \stf(M)
      = \alpha \cdot \stf(P^T N P) \\
      &= \alpha \cdot \stf\left(\begin{pmatrix}P_1^T & 0 \\ P_2^T & 1\end{pmatrix}
      \begin{pmatrix}N_1 & N_2 \\ 0 & m\end{pmatrix}
      \begin{pmatrix}P_1 & P_2 \\ 0 & 1\end{pmatrix}\right) \\
    &= \alpha \cdot \stf\left(\begin{pmatrix}P_1^T N_1 P_1 & ~~ & * \\ * & ~~ & *\end{pmatrix}\right)
    = \begin{pmatrix}\alpha \cdot P_1^T N_1 P_1 & ~~ & * \\ 0 & ~~ & *\end{pmatrix}.
\end{align*}  
Thus, $M_1 = \alpha \cdot P_1^T N_1 P_1$, so $M_1 \sim N_1$.
\end{proof}

The following may be regarded as a sort of converse to Proposition \ref{sub}.

\begin{cor}
\label{cor.stab}
Let $M,N \in \M_n(\K)$ with the same homogeneous block $L$.
If $P \in \grp_n$ is such that $\stf(M)=\alpha\cdot\stf(P^TNP)$ for some 
$\alpha \in \K^\times$,
then $P_1=\gamma Q_1$ for some $Q_1 \in \Stab(L)$ and $\gamma \in \K^\times$.
\end{cor}

\begin{proof}
This is an immediate consequence of the computation in the previous proposition.
We have $L=\alpha \cdot P_1^T L P_1$.
Let $Q_1=\sqrt{\alpha}P_1$, 
then $Q_1 \in \Stab(M_1)$ and $\gamma=\sqrt{\alpha\inv}$.
\end{proof}

In the next section we will determine equivalence classes in $\M_3(\K)$ under
sf-congruence.
By Proposition \ref{sub} and Corollary \ref{cor.stab},
we may immediately divide the matrices into distinguished classes
depending on the homogeneous blocks.
The next proposition will allow us to show sf-congruence between
matrices with the same homogeneous block, but whose column vector $M_2$ or constant $m$
are scalar multiples.

\begin{prop}
\label{sfconst}
Let $M \in G_n$ and $\gamma \in \K^\times$.
Then
  \[ M \sfsim \begin{pmatrix}M_1 & \gamma\inv M_2 \\ 0 & \gamma^{-2}m\end{pmatrix}.\]
\end{prop}

\begin{proof}
Let $I$ be the $(n-1) \times (n-1)$ identity matrix.
Let $P \in \grp_n$ with $P_1=\gamma I$ and $P_2=0$.
Then
	\[ P^TMP 
		= \begin{pmatrix}\gamma^2 M_1 & \gamma M_2 \\ 0 & m\end{pmatrix} 
		= \gamma^2 \cdot \begin{pmatrix}M_1 & \gamma\inv M_2 \\ 0 & \gamma^{-2}m\end{pmatrix}.\]
\end{proof}

Proving that standard-form congruence is an equivalence relation requires the following technical lemmas.

\begin{lem}
\label{sfofsf}
If $M \in \M_n(\K)$ and $P \in \grp_n$, 
then $\stf(P^T M P) = \stf(P^T \stf(M) P)$.
\end{lem}

\begin{proof}We have,
\begin{align*}
  \stf(P^T M P) 
    &= \stf\left(\begin{pmatrix}P_1^T & 0 \\ P_2^T & 1\end{pmatrix}\begin{pmatrix}M_1 & M_2 \\ M_3^T & m\end{pmatrix}\begin{pmatrix}P_1 & P_2 \\ 0 & 1\end{pmatrix}\right) \\
    &= \stf\left(\begin{pmatrix}P_1^T M_1 P_1 & ~~ & P_1^T M_1 P_2 + P_1^T M_2 \\ P_2^T M_1 P_1 + M_3^T P_1 & ~~ & P_2^T M_1 P_2 + P_2^T M_2 + M_3^T P_2 + m\end{pmatrix}\right) \\
    &= \begin{pmatrix}P_1^T M_1 P_1 & ~~ & P_1^T M_1 P_2 + P_1^T M_2 + (P_2^T M_1 P_1 + M_3^T P_1)^T \\ 0 & ~~ & P_2^T M_1 P_2 + P_2^T M_2 + M_3^T P_2 + m\end{pmatrix} \\
    &= \begin{pmatrix}P_1^T M_1 P_1 & ~~ & P_1^T M_1 P_2 + P_1^T M_2 + P_1^T M_1^T P_2 + P_1^T M_3 \\ 0 & ~~ & P_2^T M_1 P_2 + P_2^T M_2 + P_2^T M_3 + m\end{pmatrix} \\
    &= \stf\left(\begin{pmatrix}P_1^T & 0 \\ P_2^T & 1\end{pmatrix} \begin{pmatrix}M_1 & ~~ & M_2 + M_3 \\ 0 & ~~ & m\end{pmatrix} \begin{pmatrix}P_1 & P_2 \\ 0 & 1\end{pmatrix}\right) \\
    &= \stf(P^T \stf(M) P).
\end{align*}\end{proof}

\begin{prop}Standard-form congruence defines an equivalence relation.\end{prop}

\begin{proof}Reflexivity is obvious. Now suppose $M \sfsim M'$, so $\stf(M) = \alpha \cdot \stf(P^T M' P)$ for some $\alpha \in \K^\times$ and $P \in \grp_n$. By Lemma \ref{sfofsf},
\begin{align*}
  (P\inv)^T \stf(M) (P\inv) &=  \alpha \cdot (P\inv)^T \stf(P^T M' P) (P\inv) \\
 \stf\left((P\inv)^T \stf(M) (P\inv)\right) &= \alpha \cdot \stf\left( (P\inv)^T \stf(P^T M' P) (P\inv) \right) \\
 \alpha\inv \cdot \stf\left( (P\inv)^T M (P\inv)\right) &= \stf\left( (P\inv)^T P^T M' P (P\inv) \right) \\
 \alpha\inv \cdot \stf\left( (P\inv)^T M (P\inv)\right) &= \stf(M').
\end{align*}
Hence, $M' \sfsim M$, so symmetry holds. 

Finally, suppose $M \sfsim M'$ and $M' \sfsim M''$. Then there exists $\alpha,\beta \in \K^\times$ and $P,Q \in \grp_n$ such that \[\stf(M) = \alpha \cdot \stf(P^T M' P) \text{ and } \stf(M') = \beta \cdot \stf (Q^T M'' Q).\] By two additional applications of Lemma \ref{sfofsf},
\begin{align*}
  \stf(M) 
    &= \alpha \cdot \stf(P^T M' P) = \alpha \cdot \stf\left( P^T \stf(M') P \right) \\
    &= \alpha \cdot \stf\left( P^T \left(\beta \cdot \stf (Q^T M'' Q) \right) P \right) = (\alpha \beta) \cdot \stf( (Q P)^T M'' (QP)).
\end{align*}
Thus, $M \sfsim M''$, so transitivity holds as well.
\end{proof}

\section{Canonical Forms}
\label{canforms}

In this section, 
we determine equivalence classes for $\M_3(\K)$ under sf-congruence.
Canonical forms for these equivalence classes are presented in Theorem \ref{sfforms}.

If $M \sfsim N$, then $M_1 \sim N_1$ by Proposition \ref{sub}.
Thus, we may assume that $M \in G_3$ and $M_1$ is one of \eqref{22forms}.
By Proposition \ref{stabs}, 
it is left only to determine which pairs 
$(M_2,m)$ determine distinct forms.

\begin{cor}
\label{wamat}
Let $p,q \in \K^\times$. 
Let $W_p$ and $W_q$ be the defining matrices
for $\wap$ and $\wa$, respectively.
Then $\wap \sfsim \wa$ if and only if $p=q^{\pm 1}$.
\end{cor}

\begin{proof}
That $W_p \sfsim W_q$ if $p=q^{\pm 1}$ is an easy check and we omit it.
The converse now follows by Corollary \ref{qp-cong} and Proposition \ref{sub}.
\end{proof}

Our last step is to determine, 
for each canonical form in $\M_2(\K)$, 
which pairs $(M_2,m)$ give sf-congruent matrices.

\begin{thm}
\label{sfforms}
Canonical forms for $\M_3(\K)$ under sf-congruence are given below:
\begin{align*}
&  \left(\begin{smallmatrix}1 & 0  & 0 \\ 0  & 0 & 0 \\ 0 & 0 & 0  \end{smallmatrix}\right)
   \left(\begin{smallmatrix}1 & 0  & 0 \\ 0  & 0 & 0 \\ 0 & 0 & -1 \end{smallmatrix}\right)    
   \left(\begin{smallmatrix}1 & 0  & 0 \\ 0  & 0 & 1 \\ 0 & 0 & 0  \end{smallmatrix}\right)  
&  & \left(\begin{smallmatrix}0 & -1 & 0 \\ 1  & 1 & 0 \\ 0 & 0 & 0  \end{smallmatrix}\right)
   \left(\begin{smallmatrix}0 & -1 & 0 \\ 1  & 1 & 0 \\ 0 & 0 & 1  \end{smallmatrix}\right)
   \left(\begin{smallmatrix}0 & -1 & 1 \\ 1  & 1 & 0 \\ 0 & 0 & 0  \end{smallmatrix}\right) \\
& \left(\begin{smallmatrix}0 & 0  & 0 \\ 1  & 0 & 0 \\ 0 & 0 & 0  \end{smallmatrix}\right)  
   \left(\begin{smallmatrix}0 & 0  & 0 \\ 1  & 0 & 0 \\ 0 & 0 & -1 \end{smallmatrix}\right) 
&  & \left(\begin{smallmatrix}0 & -1 & 0 \\ q & 0 & 0 \\ 0 & 0 & 0  \end{smallmatrix}\right)
   \left(\begin{smallmatrix}0 & -1 & 0 \\ q & 0 & 0 \\ 0 & 0 & 1 \end{smallmatrix}\right)
   \left(\begin{smallmatrix}0 & -1 & 0 \\ 1  & 0 & 1 \\ 0 & 0 & 0  \end{smallmatrix}\right).
\end{align*}
Moreover, the forms involving $q$ are determined up to replacement by $q\inv$.
\end{thm}

\begin{proof}Suppose $M \in \M_3(\K)$.
We perform necessary congruence operations to put $M_1$ in canonical form.
Thus, $M$ is sf-congruent to a block matrix of the form
  \[ N = \begin{pmatrix}L & N_2 \\ 0 & n\end{pmatrix}\]
where $L$ is one of \eqref{22forms}, 
$N_2 = \begin{pmatrix}u & v\end{pmatrix}^T \in \K^2$, 
and $n \in \K$. 
Let $\displaystyle P=\begin{pmatrix}P_1 & P_2 \\ 0 & 1\end{pmatrix} \in \grp_3$.
By Corollary \ref{cor.stab}, we may assume $P_1 \in \Stab(L)$.
Write $P_1$ as in Proposition \ref{stabs}
and $P_2 = \begin{pmatrix}e & f\end{pmatrix}^T \in \K^2$. 

(Case 1: $L=M_{x^2}$) 
There are two cases for the stabilizer of $L$ corresponding to $\pm 1$. 
Both cases are similar and we only consider the positive case below,
\[\stf(P^TNP)=\begin{pmatrix} 1 & 0 & 2e+u+rv \\ 0 & 0 & sv \\ 0 & 0 & e^2+eu+fv+n \end{pmatrix}.\]
Because $\det(P)\neq 0$, then $s\neq 0$. 
Thus, $sv=0$ if and only if $v=0$. 
In case $v\neq 0$ we set 
$e=0$, $s=v\inv$, $r=-uv\inv$, and $f=-nv\inv$. 
This is the defining matrix of $\kx$. 
In case $v=0$, then we set $e=-\frac{1}{2}u$. 
The bottom right entry becomes $-\frac{1}{4}u^2+n$. 
Thus, if $n=\frac{1}{4}u^2$, 
then we have the defining matrix of $\xx$ and otherwise, 
by Proposition \ref{sfconst}, that of $\sxx$.

(Case 2: $L=M_{yx}$)
We have 
\[\stf(P^TNP)=\begin{pmatrix}0 & 0 & r(u+f) \\ 1 & 0 & r\inv(e+v) \\ 0 & 0 & fe+eu+fv+n\end{pmatrix}.\]
Setting $f=-u$ and $e=-v$ gives a bottom right entry of $n-uv$. 
Thus, there are two cases corresponding to $n=uv$ and $n\neq uv$. 
In the former case we arrive at the defining matrix of $\yx$ 
and in the other case, 
by Proposition \ref{sfconst}, that of $\os$.

(Case 3: $L=M_{\jp}$)
There are two cases for the stabilizer. 
We consider only the positive case, which gives,
	\[\stf(P^TNP)=\begin{pmatrix}0 & -1 & u \\ 1 & 1 & 2f+ru+v \\ 0 & 0 & f^2+eu+fv+n\end{pmatrix}.\]
Setting $f=-\frac{1}{2}(ru+v)$ allows us to make the 
$(2,3)$-entry zero. 
If $u=0$, we let $f=-\frac{1}{2}v$ and the 
$(3,3)$-entry becomes $n-\frac{1}{4}v^2$. 
Thus, in case $n=\frac{1}{4}v^2$ we have the defining matrix for $\jp$ 
and otherwise that for $\sjp$. 
If $u\neq 0$, then by Proposition \ref{sfconst} we can assume $u=1$.
Thus, $f^2+fv=\frac{1}{4}(r^2-v^2)$, and so we take
$e=\frac{1}{4}(v^2-r^2)-n$ so that the $(3,3)$-entry is zero,
giving the defining matrix for $\envv$.

(Case 4a: $L=M_1$)
Then
\[ \stf(P^TNP) = \begin{pmatrix}0 & -1 & au+cv \\ 1 & 0 & bu+dv \\ 0 & 0 & eu + fv + n \end{pmatrix}.\]
Suppose $u=v=0$. 
If $n=0$, then we have the defining matrix for $\K[x,y]$ 
and otherwise we have the matrix for $\fwa$. 
Suppose $u=0$ and $v \neq 0$. 
Setting $a=v$, $c=0$, $d=v\inv$, and $f=-nv\inv$
gives the defining matrix for $\env$.
Similarly for the case $v=0$ and $u \neq 0$. 
Finally, suppose $u,v \neq 0$. 
We choose $e,f$ such that $eu+fv=-n$. 
Because $\det(P_1)\neq 0$,
we can choose $a,b,c,d$ such that
$au+dv=0$ and $bu+dv=1$,
giving the defining matrix for $\env$.

(Case 4b: $L=M_q$, $q \neq \pm 1$)
We note that, in case $q=-1$, there are additional matrices in the stabilizer group than those considered here.
However, they are not needed in this result.
We have
\[ \stf(P^TNP) = \begin{pmatrix}0 & -1 & r( u+ (q-1)f)\\ q & 0 & r\inv (v+(q-1)e) \\ 0 & 0 & fe(q-1) + eu + fv + n \end{pmatrix}.\]
Set $f=u(1-q)\inv$ and $e=v(1-q)\inv$. 
Then the bottom right entry becomes $n-uv(q-1)\inv$. 
Thus, if $n=uv(q-1)\inv$, then this form corresponds to the 
defining matrix for $\qp$ and otherwise, 
by Proposition \ref{sfconst}, 
it corresponds to that of $\wa$.

Combining this with Corollaries \ref{qp-cong} and \ref{wamat} completes the result.
\end{proof}

\section{Classification}
\label{dist}

We wish to show that the list in Theorem \ref{classification} 
is complete with no isomorphic repetitions. 
The observant reader may have noticed a discrepancy 
in the Theorem \ref{sfforms} and Theorem \ref{classification}. 

\begin{prop}
\label{envv}
The algebras $\env$ and $\envv$ are isomorphic.
\end{prop}

\begin{proof}
Let $X,Y$ be the generators for $\env$ and
let $x,y$ be the generators for $\envv$.
Define a map $\Phi:\env \rightarrow \envv$ by $\Phi(X)=-y$, $\Phi(Y)=x+y^2$. 
This map extends to an algebra homomorphism since
\begin{align*}
  \Phi(Y)\Phi(X)-\Phi(X)\Phi(Y)+\Phi(Y)&=(x+y^2)(-y)-(-y)(x+y^2) + (x+y^2) \\
    &= yx-xy+x+y^2.
\end{align*}
We also define $\Psi:\envv \rightarrow \env$ by $\Psi(x) = Y-X^2$, $\Psi(y) = -X$. This map also extends to an algebra homomorphism since
\begin{align*}
  \Psi(y)\Psi(x)&-\Psi(x)\Psi(y)+\Psi(x)+\Psi(y)\Psi(y) \\
    &= (-X)(Y-X^2)-(Y-X^2)(-X)+(Y-X^2)-(-X)^2 = 0.
\end{align*}
It is readily checked that $\Psi(\Phi(X))=X$ and $\Psi(\Phi(Y))=Y$ so that $\Psi=\Phi\inv$.\end{proof}

This is the one case where two algebras are isomorphic even though their defining matrices are not sf-congruent. 
This makes sense as the map $\Phi$ constructed above is not an affine isomorphism. 
The relationship between $\env$ and $\envv$ is explored further in \cite{essreg}. 
In particular, $\env$ is a \textit{Poincar\'e-Birkhoff-Witt  (PBW) deformation} of $\K[x,y]$ while 
$\envv$ is a PBW deformation of $\jp$. 
In Section \ref{homog}, we will show that the respective 
homogenizations of $\env$ and $\envv$ are not isomorphic.

One can divide the remaining algebras into two classes: the domains and non-domains. 
The non-domains can be distinguished using well-known ring-theoretic results. 
For details, we refer the interested reader to \cite{gad-thesis}.

\begin{prop}
\label{nd-niso}
The algebras $\yx$, $\xx$, $\sxx$ and $\os$ are all non-isomorphic.
\end{prop}

\begin{proof}
The algebras $\xx$, $\sxx$ and $\os$ are prime while $\yx$ is not.
We have
\[\gld\sxx = \gld\os = 1\] 
whereas $\gld\xx = \infty$. 
Finally, $\gk\os = 2$ whereas $\gk\sxx = \infty$.
\end{proof}

The domains can be further subdivided into 
differential operator rings, quantum Weyl algebras, and quantum planes.
We review each of those classes here.

\begin{prop}
\label{wa-iso}
Let $p,q \in \K^\times$.
Then $\wa \iso \wap$ if and only if $p=q^{\pm 1}$.
\end{prop}

\begin{proof}
When $p$ and $q$ are not roots of unity,
this was proved in \cite{adcorps}, Corollary 3.11 (c).
This result follows in full from \cite{nspace},
Propositions 6.1, 6.3, and 6.4.
See also \cite{suarez} for a more general result
that also applies here.
\end{proof}

The corresponding result for quantum planes is \cite{gad-thesis}, Corollary 4.2.12.
However, as a more general result is useful in 
Section \ref{homog},
we review it here.

We say $\bq = (q_{ij}) \in\M_n(\K^\times)$ is 
\textit{multiplicatively antisymmetric} 
if $q_{ii} = 1$ and $q_{ij}=q_{ji}\inv$ for all $i \neq j$.
Let $\mathcal{S}_n$ be the symmetric group on $n$ letters.
If $A \in \M_n(\K^\times)$ is multiplicatively antisymmetric
and $\sigma \in \mathcal{S}_n$,
then $\sigma$ acts on $A$ by $\sigma.A = [A_{\sigma(i)\sigma(j)}]$.
We say $\bp$ is a permutation of $\bq$ 
if there exists $\sigma \in \mathcal{S}_n$ such that $\bp=\sigma.\bq$.

For $\bq \in \M_n(\K^\times)$ multiplicatively antisymmetric, 
\textit{quantum affine $n$-space} $\qnq$ 
is defined as the algebra with generating basis 
$\{x_i\}$, $1 \leq i \leq n$, 
subject to the relations $x_ix_j=q_{ij}x_jx_i$ for all 
$1 \leq i,j \leq n$. 

\begin{thm}[\cite{gad-thesis}, Theorem 4.2.11]
\label{qas-iso}
$\qnp \iso \qnq$ if and only if $\bp$ is a permutation of $\bq$.
\end{thm}

\begin{proof}
(Sketch)
Let $\{x_i\}$ be a generating basis for $\qnp$ and
$\{y_i\}$ that for $\qnq$.
If $\bp=\sigma.\bq$, then one easily constructs an isomorphism
$\Phi:\qnp \rightarrow \qnq$ given by $\Phi(x_i)=y_{\sigma(i)}$.

Conversely, any isomorphism determines a permutation of the 
degree one elements.
One then checks that such a permutation must also permute the
parameters accordingly.
\end{proof}

We now review isomorphisms between the rings $\env$, $\jp$, and $\sjp$.

Let $S$ be a ring. 
Given $\sigma \in \Aut(S)$, 
a $\K$-linear map $\delta:S \rightarrow S$ is said to be a 
\textit{$\sigma$-derivation} if it satisfies the twisted Leibniz rule, 
$\delta(ab)=\sigma(a)\delta(b)+\delta(a)b$ 
for all $a,b \in S$. 
The \textit{skew polynomial ring}
$R=S[x;\sigma,\delta]$ 
is the overring of $S$ with commutation given by 
$x a = \sigma(a)x + \delta(a)$ for all $a \in S$. 
If $\delta=0$, then we write $R=S[x;\sigma]$. 
If $\sigma=\id_S$, then we write $R=S[x;\delta]$ 
and $R$ is said to be a \textit{differential operator ring}.
The algebras $\env$, $\jp$, and $\sjp$ all have this form with 
$\delta(y)=y,y^2$, and $y^2+1$, respectively. 

\begin{prop}[Alev, Dumas, \cite{dumas}, Proposition 3.6]
\label{dumas}
Let $\K[y][x;\delta]$ and $\K[Y][X;d]$ be differential operator rings 
with $\delta(y)=f \in \K[y]$ and $d(Y)=g \in \K[Y]$.
Then $\K[y][x;\delta] \iso \K[Y][X;d]$ if and only if there exists 
$\lambda, \alpha \in \K^\times$ and $\beta \in \K$ such that 
$f(y) = \lambda g(\alpha y + \beta)$.
\end{prop}

\begin{cor}
\label{diff-iso}
The algebras $\fwa$, $\env$, $\jp$ and $\sjp$ are all non-isomorphic.
\end{cor}

\begin{proof}
The algebra $\fwa$ is simple and therefore distinct.
The algebra $\env$ is not isomorphic to $\jp$ and $\sjp$ since $\deg(xy-yx)=1$. 
If $\sjp \iso \jp$ is an isomorphism, 
then by Proposition \ref{dumas} there exists 
$\alpha,\beta \in \K$ and $\lambda\in \K^\times$ such that 
\[ (y^2+1)=\lambda(\alpha y + \beta)^2 
= \lambda(\alpha^2 y^2 + 2\alpha\beta y + \beta^2).\]
Comparing coefficients of $y$ we get that $\alpha=0$ or $\beta=0$,
a contradiction.
\end{proof}
 
Combining with results obtained previously, we can now complete the proof of our main theorem.

\vspace{1em}

\noindent\textit{Proof of Theorem \ref{classification}.} Let $A$ and $A'$ be of the form (\ref{form}) with defining matrices $M,M' \in \M_3(\K)$, respectively. If $M \sfsim M'$, then $A \iso A'$. By Theorem \ref{sfforms} and Proposition \ref{envv}, we need only show that there are no additional isomorphisms between the algebras in the present theorem.

The non-domains $\yx$, $\xx$, $\sxx$ and $\os$ are all non-isomorphic by Proposition \ref{nd-niso}. The algebra with defining polynomial $x^2-y$ is isomorphic to $\kx$ via the map $x \mapsto x$ and $y \mapsto x^2$. 
This algebra, along with 
$\mathcal{O}_1(\K^2) \iso \K[x,y]$ are the only commutative rings considered, and are distinct from one another.
As noted above, $\fwa$ is simple and therefore distinct from the other domains considered.

The domains can be divided, as above,
into one of three classes:
quantum planes, quantum Weyl algebras, and differential operator rings.
By Corollary \ref{diff-iso},
Proposition \ref{wa-iso},
and Theorem \ref{qas-iso},
the algebras belonging to each class are distinct from one another
with the exceptions $\qp \iso \mathcal{O}_{q\inv}(\K^2)$
and $\wa \iso A_1^{q\inv}(\K)$.
Proving that an algebra belongs to exactly 
one of these classes requires a study of their automorphism groups. 

If $q \neq \pm 1$, then 
$\Aut(\wa) \iso (\K^\times)$ and 
$\Aut(\qp) \iso (\K^\times)^2$. 
On the other hand, if $q=-1$, 
then $\Aut(\wa) \iso \K^\times \rtimes \{\omega\}$ and 
$\Aut(\qp) \iso (\K^\times)^2 \rtimes \{\omega\}$ where 
$\omega$ is the involution switching the generators $x$ and $y$ 
(see \cite{alev2} and \cite{alev1}). 
By counting subgroups of order four, 
it follows that $\wap \niso \qp$ for all $p,q \in \K^\times$. 
In particular, $\K^\times$ has one subgroup of order four and 
$(\K^\times)^2$ has four. 
On the other hand, in $\K^\times \rtimes \{\omega\}$ there are 
two subgroups of order four and in 
$(\K^\times)^2 \rtimes \{\omega\}$ there are eight.

As a consequence of Proposition \ref{dumas},
the automorphism groups of $\env$, $\jp$, or $\sjp$
are non-abelian semidirect products of a subgroup of $\K^\times$ by $\K[y]$
(see also \cite{dumas}, Proposition 3.6).
Thus, these rings are distinct from $\qp$ and $\wa$ except
perhaps in the case that $q=-1$.
However, in this case, $x^2$ is central and
$\qp$ and $\wa$ are not primitive by \cite{kirkkuz}, Proposition 3.2,
whereas differential operator rings over $\K[y]$ are always primitive (see, e.g., \cite{goodwa}).
\qed

Our results can be summed up succinctly in the following theorem. 

\begin{thm}
\label{tgequiv}
Let $A$ and $A'$ be of the form \eqref{form} 
with defining matrices $M,M' \in \M_3(\K)$, respectively. 
If $M \sfsim M'$, then $A \iso A'$. 
Conversely, if $A \iso A'$, 
then $M \sfsim M'$ unless $A\iso\env$ and $A'\iso\envv$ (or vice-versa).
\end{thm}

\section{Homogenizations}
\label{homog}

Our goal in this section is to show that sf-congruence applies to
algebras of the form \eqref{hform} and prove a result analogous to
Theorem \ref{classification} for these algebras.

\begin{thm}
\label{hclass}
Let $H$ be for the form \eqref{hform}. 
Then $H$ is isomorphic to one of the following algebras, 
with one representative of $f$ given in each case:
\begin{alignat*}{3}
  &\hqp,  f=xy-qyx ~ (q \in \K^\times),   
  	& ~ & \hwa, f=xy-qyx-z^2 ~ (q \in \K^\times), \\
  &\hjp,  f=yx-xy+y^2,           
  	& ~ & \hsjp, f=yx-xy+y^2+z^2, \\
  &\henv, f=yx-xy+yz,           
  	& ~ & \henvv, f=yx-xy+y^2+xz, \\
  &\hxx,  f=x^2,              
  	& ~ & \hsxx, f=x^2-z^2,\\
  &\hyx,  f=yx,               
  	& ~ & \hos, f=yx-z^2, \\
  &\hkx,  f= x^2+yz.           
  	& ~ & 
\end{alignat*}
Furthermore, the above algebras are pairwise non-isomorphic, except 
\[\hqp \iso H(\mathcal{O}_{q\inv}(\K^2)) 
	\text{ and } \hwa \iso H(A_1^{q\inv}(\K)).\]
\end{thm}

The key difference in this situation, 
versus that in the case of two-generated algebras, 
is that $\henv$ and $\henvv$ are non-isomorphic 
(see Proposition \ref{hdist}).

\begin{prop}
\label{notas}
The algebras $\hkx,\hxx,\hyx,\hsxx$, 
and $\hos$ are not domains.
\end{prop}

\begin{proof}
In the case of $\hxx$ and $\hyx$, this is obvious. 
In $\hkx$,
\begin{align}
\label{kxndom}
  xyz = x(-x^2) = (-x^2)x = yzx = yxz 
  \Rightarrow (xy-yx)z=0.
\end{align}
If $\hkx$ is a domain, 
then \eqref{kxndom} implies $\hkx$ is commutative. 
This is impossible since $\hkx/(z) \iso \xx$.
In $\hos$,
  \[ y(xy-z^2) = yxy - yz^2 = (yx)y-z^2y = (yx-z^2)y = 0.\]
Finally, in $\hsxx$,
  \[ (x+z)(x-z) = x^2 - xz + zx - z^2 = x^2-z^2=0.\]
\end{proof}

Let $H$ of the form \eqref{hform}.
Consider the linear map given by,
\begin{align}
\label{h.lmap}
  \phi(x) = p_{11}x + p_{12}y + p_{13}z, ~~\phi(y) = p_{21}x + p_{22}y + p_{23}z, ~~\phi(z) = z,
\end{align}
for $p_{ij} \in \K$ with $p_{11}p_{22}-p_{12}p_{21} \neq 0$.
One should compare this to the maps \eqref{sf.lmap}.
It is clear that this defines an isomorphism 
	\[ \K\langle x,y,z \mid xz-zx, yz-zy, f \rangle \rightarrow 
		\K\langle x,y,z \mid \phi(xz-zx), \phi(yz-zy), \phi(f) \rangle.\]
Moreover, it follows that, since $z \in \cnt(H)$ then $z \in \cnt(\phi(H))$.
As in the two-dimension general case, 
we do not assume these constitute all isomorphisms between algebras
of the form \eqref{hform}.
However, a consequence of Theorem \ref{hclass} is that
if two such algebras are isomorphic,
then there exists an isomorphism of the above form.

We can represent $H$ by a triple of matrices,
\begin{align}\label{xyform}
  (X,Y,Z) 
  	= \left(\begin{pmatrix}0 & 0 & 1 \\ 0 & 0 & 0 \\ -1 & 0 & 0\end{pmatrix},     
  		\begin{pmatrix}0 & 0 & 0 \\ 0 & 0 & 1 \\ 0 & -1 & 0\end{pmatrix},   
  		\begin{pmatrix}m_{11} & m_{12} & m_{13} \\ m_{21} & m_{22} & 	
  		m_{23} \\ 0 & 0 & m_{33} \end{pmatrix}\right),
\end{align}
with 
$\left(\begin{smallmatrix}
m_{11} & m_{12} \\ 
m_{21} & m_{22}\end{smallmatrix}\right)$ not the zero matrix. 
This representation follows by letting 
$\vec{x}=(x ~ y ~ z)^T$ 
so that
\begin{align*}
  xz-zx = \vec{x}^T X \vec{x}, ~~ 
  yz-zy = \vec{x}^T Y \vec{x}, ~~
  f = \vec{x}M\vec{x}^T.
\end{align*}
As in Section \ref{stfc}, 
the matrix $M$ is not uniquely determined for $f$ 
unless we fix a standard form for $M$.

Let $H'$ be another algebra of form \eqref{hform}
and suppose $\phi:H \rightarrow H'$ is a linear isomorphism given by \eqref{h.lmap}.
The algebra $H'$ is also defined by a triple, say $(X,Y,M')$.
Since $\phi(z)=z$, then we let $P \in \grp_3$ be the matrix of $\phi$.
It is too much to ask that $P^TXP=X$ and $P^TYP=Y$. 
We can still hope to preserve those relations up to linear combination.

\begin{prop}
Let $P \in \grp_3$ and let $X,Y$ be as in \eqref{xyform}. 
The matrices $X$ and $Y$ are linear combinations of 
$P^T X P$ and $P^T Y P$.\end{prop}

\begin{proof}
Write
  \[ P = \begin{pmatrix}a_1 & a_2 & a_3 \\ b_1 & b_2 & b_3 \\ 0 & 0 & 1\end{pmatrix}.\]
Let
\begin{align*}
  U = P^T X P = \begin{pmatrix}0 & 0 & a_1 \\ 0 & 0 & a_2 \\ -a_1 & -a_2 & 0\end{pmatrix} \text{ and }
  V = P^T Y P = \begin{pmatrix}0 & 0 & b_1 \\ 0 & 0 & b_2 \\ -b_1 & -b_2 & 0\end{pmatrix}.
\end{align*}
We require $r,s \in \K$ such that $rU + sV = X$. 
That is, $ra_1 + sb_1 = 1$ and $ra_2 + sb_2 = 0$. 
Since $\det P\neq 0$, 
then this system has a solution. 
Similarly, we can find $r',s' \in \K$ such that $r'U + s'V = Y$. 
Hence, $P$ fixes the commutation relations for $z$.\end{proof}

As a consequence of the previous proposition, we have that
standard form congruence preserves the commutation relations for $z$.
Thus, we extend in a natural way the map \eqref{sfmap} to this case.
Moreover, there is no loss in referring to the matrix $M$ as the 
\textit{defining matrix} of $H$. 
If $M \sfsim M'$, then the triples $(X,Y,M)$ and $(X,Y,M')$ define isomorphic algebras. 
Therefore, if $H$ is of the form \eqref{hform},
then $H$ is isomorphic to an algebra given by one of the canonical 
forms from Theorem \ref{sfforms}.

One could now prove Theorem \ref{hclass} in a manner analogous
to Theorem \ref{classification}, that is, by considering ring-theoretic
properties of the algebras.
However, we take a different approach here by considering prime
ideals.
This will allow us to apply Theorem \ref{classification} to this
situation.

Just as each domain of the form \eqref{form} 
can be represented as a skew polynomial ring over $\K[y]$, 
so can each domain of the form \eqref{hform} 
be represented as a skew polynomial ring over $\K[y,z]$ (or $\K[x,z]$). 
In these cases we can completely determine the prime ideals.
In the case that $A$ is not a domain, we can partially describe 
the prime ideals of $H$.

The following is an immediate corollary of \cite{essreg}, 
Proposition 2.5.

\begin{cor}
\label{cntrz}
\begin{align*}
  \cnt(\hjp) &= \cnt(\hsjp) = \cnt(\henv) = \cnt(\henvv) = \K[z]. \\
  \cnt(\hqp) &= \begin{cases}\K[x^n,y^n,z] & \text{ if $q$ is a primitive $n$th root of unity} \\ \K[z] & \text{ otherwise.}\end{cases} \\
  \cnt(\hwa) &= \begin{cases}\K[x^n,y^n,z] & \text{ if $q \neq 1$ is a primitive $n$th root of unity} \\ \K[z] & \text{ otherwise.}\end{cases}  
\end{align*}\end{cor}

We say an algebra $H$ of the form \eqref{hform} has 
\textit{trivial center} in case $\cnt(H)=\K[z]$. 
The proof of Theorem \ref{hclass} will be easier in this case.

\begin{lem}
\label{zprime}
Suppose $H$ is of the form \eqref{hform}. 
If $P$ is a prime ideal of $H$ with $P \cap \K[z] \neq 0$, 
then $P$ contains $z-\alpha$ for some $\alpha \in \K$.
\end{lem}

\begin{proof}
Let $g \in P \cap \K[z]$ be nonzero. 
If $g$ is not irreducible in $\K[z]$,
then $g=g_1g_2$ for some nonconstant $g_1,g_2 \in \K[z]$.
Because $\K[z]$ is central, then $g_1Hg_2 = g_1g_2H \subset P$. 
The primeness of $P$ implies $g_1 \in P$ or $g_2 \in P$. 
Hence, $P$ contains $az-b$ for some $a,b \in \K$, $a \neq 0$, 
and so contains $a\inv(az-b) = z-a\inv b$.
\end{proof}

Before proceeding to the main theorem,
we need one additional definition. 
Let $J$ be an ideal in a ring $R$ and $\sigma \in \Aut(R)$. 
Then $J$ is \textit{$\sigma$-cyclic} if $J=J_1 \cap \cdots \cap J_n$ 
where the $J_i$ are distinct prime ideals of $R$ such that 
$\sigma\inv(J_{i+1})=J_i$ and $\sigma\inv(J_1)=J_n$.

Let $R$ be a commutative ring. 
The \textit{radical} of an ideal $I$ in $R$ is 
$\sqrt{I}=\{a \in R \mid a^n \in I \text{ for some } n\}$. 
It is not difficult to see that the radical of an ideal is again an ideal in $R$. 
The ideal $I$ is said to be \textit{primary} if $ab \in I$ implies 
$a \in I$ or $b^n \in I$ for all $a,b \in R$ and some $n \in \NN$. 
If $I$ is primary, then $\sqrt{I}$ is prime.

Suppose $H$ is a domain of form \eqref{hform}. 
If $z$ is not a zero divisor in $H$, 
then we can localize at the set $C=\K[z]\backslash\{0\}$.
We refer to this ring as $H_C$. 

\begin{thm}
\label{sprimes}
Let $H$ be a domain of the form \eqref{hform}. 
If $P$ is a nonzero prime ideal in $H$, 
then one of the following holds:
    \begin{enumerate}
    \item $z \in P$ and $P$ corresponds to a prime of $H/(z)$;
    \item $z - \alpha \in P$, $\alpha \in \K^\times$, and $P$ corresponds to a prime of $H/(z-\alpha)$;
    \item $xy-yx \in P$;
    \item $P \cap \K[y,z]=(g_1\cdots g_n)$ where the $g_i$ are irreducible polynomials in $\K[y,z]$ disjoint from $\K[z]$;
    \item $P \cap \K[x,z]=(g_1\cdots g_n)$ where the $g_i$ are irreducible polynomials in $\K[x,z]$ disjoint from $\K[z]$.
    \end{enumerate}
\end{thm}
    
\begin{proof}
First, suppose $P'=P \cap \K[z]\neq 0$. 
Then $P'$ is a prime ideal of $\K[z]$ and so, by Lemma \ref{zprime}, $z-\alpha \in P$ for some $\alpha \in \K$. 
Now assume $P \cap \K[z] = 0$. 
In this case, $P$ extends to a prime ideal in $H_C$. 
Let $R=\K(z)[y]$, then $R$ has Krull dimension one. 
Let $I=P\cap R$. 
By \cite{irvingII}, Theorem 7.2, one of the following must hold:
    \begin{itemize}
    \item $H_C/P$ is commutative;
    \item $I$ is $\sigma$-cyclic for some $\sigma \in \Aut(R)$;
    \item $I$ is primary with $\sigma(\sqrt{I})=\sqrt{I}$.
  \end{itemize}
If $H_C/P$ is commutative, then $xy-yx \in P$. 
If $I$ is $\sigma$-cyclic, then $I=P_1 \cap \cdots \cap P_n$ for distinct prime ideals of $R$. 
But the prime ideals of $R$ are exactly extensions of prime ideals of $\K[y,z]$ disjoint from $\K[z]$. 
Therefore, $I = (g\sigma(g) \cdots \sigma^{n-1}(g))$ 
for some irreducible $g \in \K[y,z]$, $g \notin \K[z]$, 
$\sigma^n(g)=g$. 
Otherwise, $I$ is primary. 
Since $\sqrt{I}$ is prime, 
then $\sqrt{I}=(g)$ where $g \in R$ is irreducible. 
We claim $I = (g^n)$. 
Because $R$ is a principal ideal domain, 
$I = (h)$ for some $h \in R$. 
Write $h=tg^n$ where $n$ is maximal such that $g$ does not divide $t$. 
We claim $t$ is a constant. 
Suppose otherwise, then $g^m \notin I$ for any $m > 0$. 
This contradicts $g \in \sqrt{I}$, and so the claim holds and 
$I = (g^n)$. 
For the remaining case, 
we need only observe that we can rewrite 
$H$ as a skew polynomial ring with base ring $\K[x,z]$ and repeat.
\end{proof}

Note that the (4) and (5) occur if and only if
$H=\hqp$ or $H=\hwa$ with $q$ a primitive $n$th root of unity.
This is because in this case the center is nontrivial (Corollary \ref{cntrz}).
By Theorem \ref{sprimes}, it remains only to consider those prime ideals $P$
whose intersection with the base ring is zero.
In this case, it suffices to localize the base ring $\K[y,z]$ to 
$Q=\K(y,z)$ and consider the extension $Q[x;\sigma,\delta]$. 
Then we can appeal to \cite{lemat}, Corollary 2.3. 
Thus, the prime ideals lying over zero in $Q$ are of the form 
$(g)$ where $g \in \K[x^n,y^n,z]$ is irreducible and such that 
$g \notin \K[y^n,z]$ and $g \notin \K[x^n,z]$.

\begin{prop}
\label{hdist}
Let $A$ and $A'$ be of the form \ref{form} such that $\cnt(A)=\cnt(A)=\K$.
Let $H=H(A)$ and $H'=H(A')$, respectively. 
If $H \iso H'$ then either $A \iso A'$ and $\gr(A) \iso \gr(A')$, or else $A \iso \gr(A')$ and $A' \iso \gr(A)$.
\end{prop}
  
\begin{proof}
Suppose $\Phi:H \rightarrow H'$ is the given isomorphism. 
Let $I$ and $J$ be the ideal in $H$ generated by $z$ and $z-1$, 
respectively, and let $\Phi(I)=I',\Phi(J)=J'$. 
Since $I$ and $J$ are generated by a central element in $H$, 
then $I' \cap \K[z] \neq 0$ and similarly for $J'$. 
By Proposition \ref{zprime}, 
$I'=(z-\alpha)$ and $J'=(z-\beta)$ for some $\alpha,\beta \in \K$. 
This gives, 
\[\Phi(H/I) \iso \Phi(H)/\Phi(I) \iso H'/I'.\] 
Similarly, $\Phi(H/J) \iso H'/J'$.
\end{proof}

We observe that the previous proposition holds 
so long as the homogenizing element is regular in $H$ and $H'$.
We can now show that the non-domains are non-isomorphic. 
Details are available in \cite{gad-thesis}.

\begin{prop}
\label{ndiso}
The algebras $\hxx$, $\hsxx$, $\hyx$, $\hos$, and $\hkx$ are non-isomorphic.
\end{prop}

\begin{proof}
The algebras $\hxx, \hsxx$ and $\hos$ are prime whereas $\hkx$ and $\hyx$ are not. 
Moreover, $\gk(\hyx) = 2$ while $\hkx$ has infinite GK dimension. 
Similarly, $\gk(\hos)=2$ whereas 
\[\gk(\hxx) = \gk(\hsxx)=\infty.\]
\end{proof}

\textit{Proof of Theorem \ref{hclass}.} 
By Proposition \ref{ndiso},
we need only consider the domains.
In this case, the algebras
$\hjp$, $\hsjp$, $\henv$, $\henvv$
as well as $\hqp$ and $\hwa$ for 
$q \in \K^\times$ not a root of unity,
are all distinct as a consequence of 
Proposition \ref{hdist} and then an application
of Theorem \ref{classification}.

This leaves only the root of unity case for $\hqp$ and $\hwa$,
as well as isomorphisms between those.
That $\hqpp \iso \hqp$ if and only if $p = q^{\pm 1}$ follows
from Theorem \ref{qas-iso}.
Similar methods can be applied to show 
$\hwap \iso \hwa$ if and only if $p = q^{\pm 1}$
and $\hqpp \niso \hwa$ for all $p,q \in \K^\times$.
The reader is referred to \cite{gad-thesis} for more details.
\qed

We now summarize our results in a similar manner as to Theorem \ref{tgequiv}. 
The statement, however, is slightly more satisfying.

\begin{thm}
Let $H$ and $H'$ be of the form \eqref{hform} 
with defining matrices $M,M' \in \M_3(\K)$, respectively. 
Then $M \sfsim M'$ if and only if $H \iso H'$.
\end{thm}

\section*{Acknowledgements}

The author would like to thank his advisor, Allen Bell, 
for suggesting this topic and for his patient guidance and support throughout this project.

\bibliographystyle{plain}

\end{document}